\documentclass[fontsize=10pt]{scrreprt}

\usepackage[square,numbers,nonamebreak]{natbib}
\usepackage{setspace}

\usepackage[english]{babel}

\usepackage{chngcntr}

\usepackage[hang]{footmisc}
 at 11truept
\font\ssc=pplrc9d at 11 truept
\newcommand\qedbox{$\rlap{$\sqcap$}\sqcup$}

\usepackage[hmarginratio=1:1, bottom=2cm, top=2.5cm]{geometry}

\usepackage[automark]{scrlayer-scrpage}
\cohead{}

\let\ceheadL\cehead
\renewcommand\cehead[1]{
\ceheadL{\textnormal{#1}}
}

\lefoot{}
\rofoot{}

\usepackage{graphicx}
\usepackage[usenames,dvipsnames,svgnames,table]{xcolor}
\definecolor{Maroon}{cmyk}{0, 0.87, 0.68, 0.32}
\definecolor{RoyalBlue2}{cmyk}{80,100,0,0.1}

\newcommand\auths[1]{\large \textsc{\textcolor{Maroon}{#1}}\setstretch{1.2}}
\newcommand\titl[1]{\center \linespread{1.1}\color{RoyalBlue2}\Large\textbf{ #1}\color{black}\bigskip} 
\renewcommand\abstract[1]{
\begin{center}
{\textbf{Abstract}
}
\end{center}
{
\linespread{1.1}\fontsize{9pt}{-10pt}\selectfont #1}}

\usepackage[utf8]{inputenc}
\usepackage[sc,osf]{mathpazo}
\usepackage[euler-digits]{eulervm}
\usepackage[T1]{fontenc}
\usepackage{mathtools}

\DeclareSymbolFont{operators}{\encodingdefault}{ppl}{m}{n}
\DeclareMathAlphabet{\mathbf}{\encodingdefault}{ppl}{bx}{n}
\DeclareMathAlphabet{\mathit}{\encodingdefault}{ppl}{m}{it}

\usepackage{amsfonts}
\usepackage{amsmath}
\usepackage{ragged2e}
\usepackage{titlesec}
\usepackage{amssymb}
\usepackage{url}

\renewcommand{\thesection}{\arabic{section}}
 
\titleformat{\section}{\medskip\bigskip\normalfont\Large\bf}{\thesection}{0.5em}{}
\titleformat{\subsection}{\smallskip\bigskip\normalfont\large\bf}{\thesubsection}{0.5em}{}

\usepackage{amsthm}
\newtheoremstyle{dotless}{}{}{\itshape}{}{\bfseries}{}{1em}{}

\newtheorem*{theo*}{Theorem}
\newtheorem*{rem*}{Remark}
\theoremstyle{dotless}
\newtheorem{theo}{Theorem}

\newtheorem{lem}[theo]{Lemma}

\newtheorem{cor}[theo]{Corollary}

\newtheorem{ex}[theo]{Example}
\newtheorem{quest}[theo]{Question}
\renewenvironment{proof}{\smallbreak\noindent {\sc Proof \;---\;}}{\hfill\qedbox}

\counterwithout{theo}{section}
\numberwithin{theo}{section}

\DeclareOldFontCommand{\rm}{\normalfont\rmfamily}{\mathrm}
\DeclareOldFontCommand{\sf}{\normalfont\sffamily}{\mathsf}
\DeclareOldFontCommand{\tt}{\normalfont\ttfamily}{\mathtt}
\DeclareOldFontCommand{\bf}{\normalfont\bfseries}{\mathbf}
\DeclareOldFontCommand{\it}{\normalfont\itshape}{\mathit}
\DeclareOldFontCommand{\sl}{\normalfont\slshape}{\@nomath\sl}
\DeclareOldFontCommand{\sc}{\normalfont\scshape}{\@nomath\sc}

\usepackage{comment}

\usepackage{soul}
\DeclareSymbolFont{newfont}{OML}{cmm}{m}{it}%
\DeclareMathSymbol{\Varrho}{3}{newfont}{37}

\begin{document}

\titl{A note on right-nil and strong-nil skew braces\footnote{All authors are members of the non-profit association ``Advances in Group Theory and Applications'' (www.advgrouptheory.com). The second and fourth authors are supported by GNSAGA (INdAM). The second author has been supported by a research visiting grant issued by the Istituto Nazionale di Alta Matematica (INdAM). The fourth author is funded by the European Union - Next Generation EU, Missione 4 Componente 1 CUP B53D23009410006, PRIN 2022- 2022PSTWLB - Group Theory and Applications.}}

\auths{A. Ballester-Bolinches -- M. Ferrara -- V. P\'erez-Calabuig -- M. Trombetti}

\thispagestyle{empty}
\justify\noindent
\setstretch{0.3}
\abstract{Nilpotency concepts for skew braces are among the main tools with which we are nowadays classifying certain special solutions of the Yang--Baxter equation, a consistency equation that plays a relevant role in quantum statistical mechanics and in many areas of mathematics. In this context, two relevant questions have been raised in \cite{Cedo} (see Questions 2.34 and 2.35) concerning right- and central nilpotency. The aim of this short note is to give a negative answer to both questions: thus, we show that a finite strong-nil  brace $B$ need not be right-nilpotent. On a positive note, we show that there is one (and only one, by our examples) special case of the previous questions that actually holds. In fact, we show that if~$B$ is a skew brace of nilpotent type and $b\ast b=0$ for all $b\in B$, then $B$ is centrally nilpotent.}

\setstretch{2.1}
\noindent
{\fontsize{10pt}{-10pt}\selectfont {\it Mathematics Subject Classification \textnormal(2020\textnormal)}: 16T25, 20F18, 16N99}\\[-0.8cm]

\noindent 
\fontsize{10pt}{-10pt}\selectfont  {\it Keywords}: skew brace, right-nil skew brace, right nilpotent skew brace, strong-nil, centrally nilpotent skew brace\\[-0.8cm]

\setstretch{1.1}
\fontsize{11pt}{12pt}\selectfont

\bigskip\bigskip\bigskip

\section{Introduction}
The Yang-Baxter equation (YBE) is a consistency equation that was independently set by Yang \cite{yang} and Baxter \cite{baxter} in the field of quantum statistical mechanics. The study of its solutions not only has many relevant interpretations in the realm of mathematical physics but it also plays a key role in the foundation of quantum groups and furnishes a multidisciplinary approach from a wide variety of areas such as Hopf algebras,  knot theory and braid theory among others (see \cite{drinfeld90},\cite{Faddev},\cite{Gateva-Ivanova18-advmath}).

In the inspiring paper of Drinfel'd \cite{drinfeld}, the attention is drawn to the so-called set-theoretical solutions of the YBE, a family of solutions that has been the object of a very prolific research since then. A (finite) \emph{set-theoretic solution} of the YBE is a pair~\hbox{$(X,r)$,} where~$X$ is a (finite) set and $r\colon X \times X \rightarrow X \times X$ is a bijective map satisfying the equality $r_{12}r_{23}r_{12} = r_{23}r_{12}r_{23}$, where $r_{12} = r\times \operatorname{id}_X$ and $r_{23} = \operatorname{id}_X \times\, r$.  The problem of classifying all set-theoretical solutions has brought to light the definitions of new algebraic structures so that properties of solutions can be translated in terms of such structures and vice-versa. In this light, left skew braces play a key role in the classification problem of {\it non-degenerate} solutions, i.e. set-theoretic solutions for which both components are bijective (see~\cite{GuarnieriVendramin17} for example).
 
A \emph{\textnormal(left\textnormal) skew brace} $B$ is a set endowed with two group structures, $(B,+)$ and $(B,\cdot)$, satisfying the following {\it skew} distributivity property $$a \cdot (b+c) = a\cdot b - a + a\cdot c\quad \forall\, a,b,c\in B.$$ If $(B,+)$ satisfies some property $\mathfrak{X}$ (such as abelianity), we say that $B$ is a skew brace of~{\it $\mathfrak{X}$ type}; in particular, skew braces of abelian type are just  Rump's \emph{\textnormal(left\textnormal) braces} (see~\cite{GuarnieriVendramin17} and~\cite{Rump07}); note that both operations in $B$ can be related by the so-called {\it star product}: $a\ast b = -a + a\cdot b -b$, for all $a,b\in B$. A non-degenerate set-theoretic solution of the~YBE naturally leads to a  skew brace structure over the group (see \cite{Rump07})
\[G(X,r) = \langle x \in X\,|\, xy = uv,\, \text{if $r(x,y) = (u,v)$}\rangle,\]  which is usually called the \emph{structure}  skew brace of~$(X,r)$. Conversely, every skew brace~$B$ defines a solution $(B,r_B)$ of the YBE (see~\cite{GuarnieriVendramin17}).





Although it is very difficult to understand arbitrary non-degenerate set-theoretic solutions of the~YBE, it turns out that nilpotency concepts for skew braces allows us to define certain particularly tamed classes of solutions. For example, right-nilpotent skew braces have been introduced to study the so-called {\it multipermutation} solutions, that is, non-degenerate set-theoretic solutions that can be retracted into the trivial solutions over a singleton after finitely many identification steps (see \cite{Cedo,cameron,periodici}); and it turns out that a non-degenerate set-theoretic solution is multipermutation if the structure skew brace associated with this solution is right-nilpotent and of nilpotent type (see \cite{Cedo}, Lemma 2.16, and \cite{55}, Theorem 4.13). Thus, understanding if right-nilpotency is equivalent to some weaker concept could have a breakthrough effect on classifying the multipermutation solutions of the YBE, since it would mean we can recognize right-nilpotency with less computational effort. This is what led the authors of~\cite{55} to ask about the equivalence of right-nilpotency with the seemingly weaker concept of right-nil (see next section for the precise definitions).

\begin{quest}[see \cite{Cedo}, Question 2.34]\label{quest1}
$\,$
\begin{itemize}
    \item Let $B$ be a finite right-nil skew brace. Is $B$ right-nilpotent?
\end{itemize}
\end{quest}

\medskip

Another nilpotency concept that plays a major role in the theory of skew braces and the solutions of the YBE is that of central nilpotency. Central nilpotency is actually the strongest nilpotency concept for skew braces and was introduced by using a skew brace-theoretical analog of the centre of a group (see~\cite{Bonatto} and~\cite{periodici}), so it can be regarded as the true analog of the usual nilpotency for groups. As for right-nilpotency, it is relevant to understand if central nilpotency can be derived from the seemingly weaker concept of strong-nil.

\begin{quest}[see \cite{Cedo}, Question 2.35]\label{quest2}
$\,$
\begin{itemize}
    \item Let $B$ be a finite strong-nil skew brace. Is $B$ centrally nilpotent?
\end{itemize}
\end{quest}

Our answers to both Questions \ref{quest1} and \ref{quest2} is negative and in fact we provide an example of a finite brace $B$ of order $32$ that is strong-nil but not right-nilpotent (see Example \ref{exB}); keep in mind that strong-nil (resp. central nilpotency) implies right-nil (resp. right-nilpotency). In this example, it happens that $(b\ast b)\ast b=0=b\ast(b\ast b)$ for every $b\in B$ but $\operatorname{Soc}(B)=\{0\}$. Thus, the only special circumstance in which the above questions could be given a positive answer is that in which $b\ast b=0$ for all $b\in B$. This is in fact happening, as shown by our main result.

\medskip

\noindent{\bf Main Theorem}\quad {\it Let $B$ be a finite skew brace of nilpotent type such that $b\ast b=0$ for all~\hbox{$b\in B$.} Then $B$ is centrally nilpotent.}

\medskip

This result can also be considered as a partial extension of Smoktuno\-wicz~\cite{Smok},~The\-o\-rem~12. We further show that in our main theorem, the nilpotency of the additive group cannot be replaced by weaker concepts, such as that of supersoluble group (see~Example~\ref{exA}).

\section{Preliminaries and results}



Let $(B,+,\cdot)$ be a skew brace. The common identity element of both group operations is denoted by $0$, and the product of two elements will be denoted by juxtaposition; as usual, group addition follows group product in the order of operations. A \emph{skew sub-brace} of a skew brace is a subgroup of the additive group which is also a subgroup of the multiplicative group. 

As we already noted in the introduction, both operations in $B$ can be related by the so-called star product ($a\ast b = -a + a\cdot b -b$, for all $a,b\in B$), which always comes first in the order of operations. Indeed, both group operations coincide if and only if $a\ast b=0$ for all $a,b\in B$; in this case, $B$ is said to be a \emph{trivial skew brace}. The following properties are essential to our work:$$
\begin{array}{c}
(ab) \ast c  =  a \ast (b\ast c) + b\ast c + a \ast c,\\[0.2cm]
ab  =  a + a \ast b + b,\\[0.2cm]
a \ast (b+c)  = a \ast b + b + a \ast c - b,
\end{array}$$ for all $a,b,c\in B$. If $X$ and $Y$ are subsets of $B$, then $X \ast Y$ is the subgroup of~$(B,+)$ generated by the elements of the form $x \ast y$, for all $x \in X$ and $y \in Y$.

For every $a\in B$, the map 
$\lambda_a \colon B \rightarrow B,$ given by 
$\lambda_a(b) = -a + ab$, is an automorphism of $(B,+)$ and the map $\lambda\colon (B,\cdot) \rightarrow \operatorname{Aut}(B,+)$ which maps $a$ to $\lambda_a$ is a group homomorphism. For every $a,b\in B$,
\[ a \ast b := \lambda_a(b) - b = - a + ab - b.\] \emph{Left-ideals} are $\lambda$-invariant skew sub-braces, or equivalently skew sub-braces $L$ such that $B \ast L \subseteq L$. A left-ideal $S$ is said to be a \emph{strong left-ideal} if $(S,+)$ is a normal subgroup of $(B,+)$, and an \emph{ideal} if $(S,\cdot)$ is also a normal subgroup of $(B,\cdot)$, or equivalently $S \ast B \subseteq S$. Ideals of skew braces allow us to take quotients in a skew brace: if $I$ is an ideal of $B$, then $B/I = \{bI = b+I: b \in B\}$ denotes the quotient of $B$ over~$I$. It should also be remarked that, for each skew sub-brace~$S$ and each strong left ideal $I$ of $B$, we have $SI = S+I$. Furthermore, for the sake of simplicity, we introduce the following notations (here,~$E$ is a subset of $B$):
\begin{itemize}
    \item To denote that $C$ is a skew sub-brace of $B$, we write $C\leq B$. To denote that~$I$ is an ideal of $B$, we write $I\trianglelefteq B$.
\end{itemize}

If $B$ is any skew brace, then the following subsets are always ideals of $B$:\[
B\ast B,\quad\operatorname{Soc}(B)  = \operatorname{Ker}(\lambda)\cap Z(B,+)\quad\textnormal{and}\quad\zeta(B)=\operatorname{Soc}(B)\cap Z(B,\cdot),\] where $Z(B,+)$ and $Z(B,\cdot)$ are the centers of $(B,+)$ and $(B,\cdot)$, respectively.

\medskip

Now, we introduce the nilpotency concepts we deal with. Let $B$ be a skew brace. Set $R_0(B)=B=L_0(B)$ and recursively define $$R_{n+1}(B)=R_{n}(B)\ast B\quad\textnormal{and}\quad L_{n+1}(B)=B\ast L_n(B)$$ for all $n\in\mathbb{N}$. Then $B$ is {\it right-nilpotent} (resp.~{\it left-nilpotent}) if $R_m(B)=\{0\}$ (resp. $L_m(B)=\{0\}$) for some $m\in\mathbb{N}$; the smallest such an $m$ is the {\it $r$-class} (resp. {\it $l$-class}) of~$B$. Thus, if $B$ is right-nilpotent of $r$-class $m$, then \[
\begin{array}{c}\label{star}
\underbrace{(\ldots((b\ast b)\ast \ldots )\ast b)}_{m\;\textnormal{times}}=0\tag{$\star$}
\end{array}
\] for all $b\in B$; similarly, if $B$ is left-nilpotent of $l$-class $m$, then \[
\begin{array}{c}\label{circle}
\underbrace{(b\ast(\ldots \ast(b\ast b))\ldots )}_{m\;\textnormal{times}}=0\tag{$\bullet$}
\end{array}
\] for all $b\in B$. Seemingly weaker concepts can be introduced if we only require \eqref{star} or~\eqref{circle} to hold (see \cite{Cedo} and \cite{Smok}). Thus a skew brace is {\it right-nil} (resp. {\it left-nil}) if, for all $b\in B$, there is a suitable $m=m_b\in\mathbb{N}$ such that equation \eqref{star} (resp. \eqref{circle}) holds.

In order to deal with right-nilpotency, we need the following chain of ideals of a skew brace $B$. Let $\operatorname{Soc}_0(B)=\{0\}$ and recursively define $\operatorname{Soc}_{n+1}(B)$ to satisfy the equality $\operatorname{Soc}_{n+1}(B)/\operatorname{Soc}_n(B)=\operatorname{Soc}\big(B/\operatorname{Soc}_n(B)\big)$. It has been proved in \cite{Cedo}, Lemma 2.16, that a skew brace $B$ of nilpotent type is right-nilpotent if and only if $B=\operatorname{Soc}_n(B)$ for some $n\in\mathbb{N}$.

A stronger concept of nilpotency is given by central nilpotency. Let $\zeta_0(B)=\{0\}$ and recursively define $\zeta_{n+1}(B)$ to satisfy $\zeta_{n+1}(B)/\zeta_n(B)=\zeta\big(B/\zeta_n(B)\big)$. A skew brace is {\it centrally nilpotent} if there is some $m\in\mathbb{N}$ for which $B=\zeta_m(B)$. We refer to \cite{tutti23} for further information about centrally nilpotent skew braces. Here, we only observe that centrally nilpotent implies both left- and right-nilpotency, and that, conversely, left- and right-nilpotency imply central nilpotency if the skew brace is of nilpotent type (see~Co\-rol\-la\-ry 2.15 of \cite{periodici}).

\begin{lem}\label{abelian}
Let $B$ be a finite skew brace of nilpotent type such that $(B,\cdot)$ is nilpotent and $b\ast b=0$ for all $b\in B$. Then~$B$ is centrally nilpotent.
\end{lem}
\begin{proof}
Let $L_p$ be any Sylow $p$-subgroup of $(B,+)$ for some prime $p$. Since $L_p$ is a characteristic subgroup of $(B,+)$ it follows that $L_p$ is a left-ideal of~$B$. Thus,~$L_p$ is also a~Sy\-low~\hbox{$p$-sub}\-group of $(B,\cdot)$ and hence $L_p$ is an ideal of $B$. It follows that~\hbox{$B=\operatorname{Dr}_{p\in\mathbb{P}}L_p$.} Thus, in order to prove that~$B$ is centrally nilpotent, we may assume $B$ has prime power order $p^n$. Since the natural semidirect product $[(B,+)]_\lambda(B,\cdot)$ is a finite $p$-group, it follows that there is an element $a\in Z(B,+)$ such that $b\ast a=0$ for all $b\in B$. Now, $$
\begin{array}{c}
0=(b+a)\ast(b+a)=(b\cdot a)\ast(b+a)=b\ast (a\ast(b+a))+a\ast(b+a)+b\ast(b+a)\\[0.2cm]
=b\ast (a\ast b)+a\ast(b+a)+b\ast(b+a)=b\ast(a\ast b)+a\ast b
\end{array}
$$ for all $b\in B$. By The\-o\-rem~4.8 of \cite{Cedo}, $B$ is left-nilpotent. Let $c=a\ast b$. Then $b\ast c=-c$ and consequently $b\ast(-c)=c$. Therefore $$\underbrace{b\ast \big(b\ast\ldots \ast(b}_{\ell\,\textnormal{times}}\ast\, c)\big)=(-1)^\ell c$$ for every $\ell\in\mathbb{N}$, which means (by left-nilpotency) that $c=a\ast b=0$. Therefore $a\in\operatorname{Ker}(\lambda)$ and hence $a\in\operatorname{Soc}(B)$. By induction on the order of~$B$, we have that $B=\operatorname{Soc}_m(B)$ for some $m\in\mathbb{N}$. Thus, Lemma 2.16 of \cite{Cedo} shows that~$B$ is right-nilpotent. Since $B$ is both left- and right-nilpotent, it is also centrally nilpotent by Co\-rol\-la\-ry~2.15 of \cite{periodici}.
\end{proof}


\begin{cor}\label{corabe}
Let $B$ be a finite brace such that $b\ast b=0$ for all $b\in B$. Then~$B$ is centrally nilpotent.
\end{cor}
\begin{proof}
It follows from Theorem 12 of \cite{Smok} that $B$ is left-nilpotent. Then The\-o\-rem~4.8 of \cite{Cedo} shows that $(B,\cdot)$ is nilpotent. Finally, Lemma \ref{abelian} completes the proof.~\end{proof}

\begin{cor}\label{corabe2}
Let $B$ be a finite skew brace of prime power order such that $b\ast b=0$. Then $B$ is centrally nilpotent.
\end{cor}


\noindent{\sc Proof of the Main Theorem} --- For any prime $p$, let $L_p$ be the Sylow $p$-subgroup of~$(B,+)$. Since $(B,+)$ is nilpotent, it follows that $(B,+)$ is the direct product of its~Sy\-low~\hbox{$p$-sub}\-groups. By~Co\-rol\-la\-ry~\ref{corabe}, we may assume there is some prime $p$ for which~$(L_p,+)$ is non-abelian. Let $Z=Z(L_p,+)$. Now, $Z\times Q$ is a strong left-ideal of $B$, where~$Q$ is the Hall~\hbox{$p'$-sub}\-group of $(B,+)$. By induction, $Z\times Q$ is centrally nilpotent, so, as a skew brace, it is the direct product of its additive Sylow subgroups (which are ideals), and in particular $Z$ and $Q$ are ideals of $Z\times Q$.

Furthermore, $L_p$ (which is also a strong left-ideal of $B$) is centrally nilpotent by~Corol\-la\-ry~\ref{corabe2}, and hence $\zeta(L_p)$ is a non-zero ideal of $L_p$ contained in $Z$.  Thus, $\zeta(L_p)$ is an ideal of $Z\times Q$. Since $Z\times Q$ is centrally nilpotent, we have that $$C=\zeta(L_p)\cap\zeta(Z\times Q)\neq\{0\}.$$ Let $c\in C$. Then $c\in Z(B,+)\cap Z(B,\cdot)$. Moreover, since any element of $B$ can be written as a sum of an element of $P$ and an element of $Q$, it follows that $c\in\operatorname{Ker}(\lambda)$. Therefore $c\in\zeta(B)$. By induction $B/\langle c\rangle$ is centrally nilpotent, and so $B$ is centrally nilpotent as well.\hfill\qedbox

\medskip

\begin{ex}\label{exA}
There exists a skew brace $B$ of order $6$ such that $b\ast b=0$ for all $b\in B$, but $B$ is not right-nilpotent.
\end{ex}
\begin{proof}
Let $(B,+)\simeq\operatorname{Sym}(3)$ and consider a product in $B$ given by $ab = b+a$ for every $a,b\in B$. It turns out that $(B,+,\cdot)$ is a skew brace such that $a \ast a = 0$ for every~\hbox{$a \in B$.} Moreover, $\lambda_a(b) = -a +b +a$ for every $a,b\in B$. Thus, $B \ast B = \langle c \rangle_+$ is the~Sy\-low~\hbox{$3$-sub}\-group of $(B,+)$. Since $B^{(3)} = (B\ast B) \ast B = B \ast B$, it follows that $B$ is not right-nilpotent.
\end{proof}

\begin{ex}\label{exB}
There exists a brace $B$ of order $32$ such that $(b\ast b)\ast b=0$ and $b \ast (b\ast b) = 0$ for all $b\in B$, but $\operatorname{Soc}(B)=\{0\}$, so, in particular, $B$ is not right-nilpotent.
\end{ex}
\begin{proof}
Let $(B,+) = \langle a \rangle \times \langle b \rangle \times \langle c \rangle \times \langle d \rangle \times \langle e\rangle$ and
\begin{align*}
(C, \cdot)  = \Big\langle m_1,m_2,m_3,m_4,m_5 \Big|
\begin{tabular}{l}
  $m_i^2 = 1,\ 1\leq i \leq 5,\ (m_5m_2)^2 = (m_5m_3)^2 = 1$,\\
  $m_5m_1m_5 = m_1m_3,\, m_5m_4m_5 = m_2m_4$
\end{tabular}\Big\rangle
\end{align*}
be groups of order $32$ respectively isomorphic to $C_2 \times C_2 \times C_2\times C_2 \times C_2$ and to a semidirect product of the form $[C_2\times C_2 \times C_2\times C_2]C_2$. We have that $C$ acts on $B$ by means of the action $\lambda \colon C \rightarrow \operatorname{Aut}(B,+)$ defined by 
\begin{align*}
\lambda_{m_1}(a) & = a, & \lambda_{m_2}(a) &= c+d+e, & \lambda_{m_3}(a) & = c+d+e, \\
\lambda_{m_1}(b) & = b, & \lambda_{m_2}(b) & = b, & \lambda_{m_3}(b) & = a+b+c+d+e,\\
\lambda_{m_1}(c) & = c, & \lambda_{m_2}(c) & = c, & \lambda_{m_3}(c) & = a+d+e, \\
\lambda_{m_1}(d) & = a+c+e, & \lambda_{m_2}(d) & = a+c+e, & \lambda_{m_3}(d) & = d,\\
\lambda_{m_1}(e) & = a+c+d, & \lambda_{m_2}(e) & = e , &  \lambda_{m_3}(e) & = e,
\end{align*}
\begin{align*}
\lambda_{m_4}(a) & = a, & \lambda_{m_5}(a) & = a+b+c+e,\\
\lambda_{m_4}(b) & = b, & \lambda_{m_5}(b) & = a+b+c+d, \\
\lambda_{m_4}(c) & = a+d+e, & \lambda_{m_5}(c) & = a+d, \\
\lambda_{m_4}(d) & = a+c+e, & \lambda_{m_5}(d) & = a+b+e,\\
\lambda_{m_4}(e) & = e, & \lambda_{m_5}(e) & = e.
\end{align*}

Consider the semidirect product $G = [B]C$ associated with this action (here, $B$ is written multiplicatively for the sake of notation). Then $G$ is trifactorised, as there exists $D = \langle abm_1, em_2, abdem_3, abcdem_4, bcm_5 \rangle \leq G$ such that $G = DC = BD$ and \hbox{$C\cap D = B\cap D = \{1\}$.} By \cite[Lem\-ma~3.2]{BallesterEsteban22}, there exists a bijective $1$-cocycle $\delta\colon C \rightarrow (B,+)$, with respect to $\lambda$, given by \hbox{$D = \{\delta(c)c: c \in C\}$} (see~Table~\ref{tab:3rnilp-nornilp}). This yields a product in $B$, provided by $bc = \delta(\delta^{-1}(b)\delta^{-1}(c))$ (see \cite{BallesterEsteban22} for further details), and we obtain a brace $(B,+,\cdot)$ corresponding to \texttt{SmallBrace(32, 24952)} in the \textsf{Yang--Baxter} library~\cite{VendraminKonovalov22-YangBaxter-0.10.2} for~\textsf{GAP}~\cite{GAP4-12-2}.
\begin{table}[h]
  \[{\small
    \begin{array}{llll}
      x&\delta(x)&x&\delta(x) \\\hline
    1&0            &  m_1m_2m_3 & a+c+e   \\
    m_1 & a+b        &  m_1m_2m_4 & a+e    \\
    m_2 & e   &  m_1m_2m_5 & d     \\
    m_3 & a+b+d+e &  m_1m_3m_4 & b+d+e  \\
    m_4 & a+b+c+d+e &  m_1m_3m_5 & b+c+d+e  \\
    m_5 & b+c    &  m_1m_4m_5 & a+b+c  \\
    m_1m_2 & b+c+d   &  m_2m_3m_4 & a+d  \\
    m_1m_3 & d+e  &  m_2m_3m_5 & a+c+d    \\
    m_1m_4 & c+d+e      & m_2m_4m_5   & c+e \\
    m_1m_5 & a+c  &  m_3m_4m_5   & b \\
    m_2m_3 & a+b+d  & m_1m_2m_3m_4& a+b+c+e \\ 
    m_2m_4 & a+b+c+d&  m_1m_2m_3m_5& a+d+e \\
    m_2m_5 & b+c+e &  m_1m_2m_4m_5& b+d \\
    m_3m_4 & a+d+e  & m_1m_3m_4m_5& a \\
    m_3m_5 & a+c+d+e  & m_2m_3m_4m_5& b+e \\
    m_4m_5 & c & m_1m_2m_3m_4m_5& c+d\\\hline
  \end{array}}
  \]
  \caption{Associated bijective $1$-cocycle}
  \label{tab:3rnilp-nornilp}
\end{table}

Note that $\delta(x)\ast \delta(x) = 0$ for every $x\in C$ of order $2$, while routine calculations show that for the rest of non-trivial elements $y\in B$, one has $y\ast y \neq 0$ but $(y\ast y)\ast y = y \ast (y\ast y) =  0$.

Now, every element of $C$ can be written as $m_1^{\epsilon_1}m_2^{\epsilon_2}m_3^{\epsilon_3}m_4^{\epsilon_4}m_5^{\epsilon_5}$, with $\epsilon_i \in \{0,1\}$, for every $1\leq i \leq 5$. If $1\neq m_1^{\epsilon_1}m_2^{\epsilon_2}m_3^{\epsilon_3}m_4^{\epsilon_4}m_5^{\epsilon_5} \in \operatorname{Ker} \lambda$, then 
\[ \lambda_{m_1^{\epsilon_1}m_2^{\epsilon_2}m_3^{\epsilon_3}m_4^{\epsilon_4}m_5^{\epsilon_5}}(e) = e\]
and so $\epsilon_1 = 0$. Moreover, $\lambda_{m_2}$ fixes $b$ and $c$, $\lambda_{m_3}$ fixes $d$ and $\lambda_{m_4}$ fixes $a$ and $b$. Thus, $\operatorname{Ker}\lambda \cap \langle m_2,m_3,m_4\rangle =\{1\}$. Since $m_2m_5 = m_5m_2$, $m_3m_5 = m_5m_3$ and $m_4m_5 = m_5m_2m_4$ it is easy to check that also $\langle m_2,m_3,m_4,m_5 \rangle \cap \operatorname{Ker}(\lambda) = \{1\}$. Hence, \hbox{$\operatorname{Ker} \lambda =\{1\}$} and $\operatorname{Soc}(B) =\{0\}$. Since $(B,+)$ is abelian, so $B$ is not right-nilpotent.~\end{proof}

\section*{Acknowledgements}

We thank the referee for their careful reading of the manuscript and for pointing out a mistake in an earlier version of the proof of Lemma \ref{abelian}.

\begin{flushleft}
\rule{8cm}{0.4pt}\\
\end{flushleft}

{
\sloppy
\noindent
Maria Ferrara

\noindent
Dipartimento di Matematica e Fisica

\noindent
Università degli Studi della Campania  ``Luigi Vanvitelli''

\noindent
viale Lincoln 5, Caserta (Italy)

\noindent
e-mail: maria.ferrara1@unicampania.it
}

\bigskip
\bigskip

{
\sloppy
\noindent
Adolfo Ballester-Bolinches, Vicent P\'erez-Calabuig

\noindent
Departament de Matem\`atiques

\noindent
Universitat de Val\`encia, Dr.\ Moliner, 50, 46100 Burjassot, Val\`encia (Spain)

\noindent
Adolfo.Ballester@uv.es;\;\; Vicent.Perez-Calabuig@uv.es

}

\bigskip
\bigskip

{
\sloppy
\noindent
Marco Trombetti

\noindent 
Dipartimento di Matematica e Applicazioni ``Renato Caccioppoli''

\noindent
Università degli Studi di Napoli Federico II

\noindent
Complesso Universitario Monte S. Angelo

\noindent
Via Cintia, Napoli (Italy)

\noindent
e-mail: marco.trombetti@unina.it 

}

\end{document}